\documentclass[12pt]{amsart}  
\usepackage{amsmath,amstext,amsthm,amssymb,amsxtra}
\usepackage[top=1.5in, bottom=1.5in, left=1.25in, right=1.25in]	{geometry}
\usepackage{txfonts} 
\usepackage[T1]{fontenc}
\usepackage{lmodern}

 \usepackage{euler}   
\usepackage{pdfsync}

\usepackage{float} 

\usepackage{tikz}

\usepackage{mathtools}
\mathtoolsset{showonlyrefs,showmanualtags}

\usepackage{hyperref} 
\hypersetup{
    colorlinks=true,       
    linkcolor=blue,          
    citecolor=magenta,        
    filecolor=magenta,      
    urlcolor=cyan           
}

\usepackage[msc-links]{amsrefs}  
\newcommand{\norm}[1]{\ensuremath{\left\|#1\right\|}}
\newcommand{\abs}[1]{\ensuremath{\left\vert#1\right\vert}}

\newcommand{\avg}[1]{\langle #1 \rangle}
\newcommand{\one}{\mathbbm{1}}
\newcommand{\eps}{\varepsilon}

\newcommand{\m}[1]{\mathcal{#1}}
\newcommand{\md}{\mathcal{D}}

\newcommand{\unit}{1\!\!1}
\renewcommand{\one}{1\!\!1}

{\end{list}}

\numberwithin{equation}{section}

\newtheorem{thm}{Theorem}[section]
\newtheorem{lm}[thm]{Lemma}

\newtheorem*{prop*}{Proposition}

\theoremstyle{remark}

\newtheorem*{rem*}{Remark}

\numberwithin{equation}{section}

\title[Two--Weight Bounds for Sparse Operators]{Entropy Bumps and Another Sufficient Condition for the Two--Weight
  Boundedness of Sparse Operators}

\author[R. Rahm]{Robert Rahm}
\address{Robert Rahm, School of Mathematics\\ Washington University in St. Louis
    \\ One Brookings Drive\\ St. Louis, MO USA 63112}
\email{robertrahm@gmail.com}

\author[S. Spencer]{Scott Spencer}
\address{Scott Spencer, School of Mathematics\\ Georgia Institute of Technology\\ 686 Cherry 
    Street\\ Atlanta, GA USA 30332-0160}
\email{spencer@math.gatech.edu}

\begin{document}

\begin{abstract}
In this short note, we give a very efficient proof of a recent 
result of Treil--Volberg and Lacey--Spencer giving sufficient conditions 
for the two--weight boundedness of a sparse operator. We also give a 
new sufficient condition for the two--weight boundedness of a sparse operator. 
We make critical use of a formula of Hyt\"onen in \cite{Hyt}.
\end{abstract}

\maketitle

\section{Introduction}
%
%
Let $\mathcal{D}$ be a dyadic lattice. Recall that a collection $\mathcal{S}$ of 
cubes in $\mathcal{D}$ is said to be sparse if the following holds, 
uniformly over $P\in\mathcal{D}$: $\abs{\cup_{Q\in\mathcal{S}:S\subset P}Q}
\leq \frac{1}{2}\abs{P}$. This implies that the following holds, uniformly over 
all $P\in\mathcal{D}$:
$
\sum_{Q\in\mathcal{S}:S\subset P}\abs{Q}
\lesssim \abs{P}.
$ For a cube $Q\in\mathcal{S}$, let $E_Q:=Q\setminus \cup_{S\in\mathcal{S}:S\subset Q} S$
and note that $\abs{E_Q}\simeq\abs{Q}$. 
The sparse operator $T_\mathcal{S}$ indexed over cubes in $\mathcal{S}$ is defined by
$
T_\mathcal{S}f(x):=\sum_{S\in\mathcal{S}}\avg{f}_{S}\unit_S(x).
$
Here and below, $\avg{f}_{Q}:=\abs{Q}^{-1}\int_{Q}f(x)dx$.

Due to deep and important theorems of Lerner, Lacey, and Rey and Conde--Alonso 
\cites{Ler,Lac,ConRey} important operators in harmonic analysis (for example, maximal 
functions, Calder\'on--Zygmund Operators, Haar shifts) are point wise 
dominated by finite sums of sparse operators. Thus, proving two--weight inequalities 
for these sparse operators will imply the same theorems for other operators of 
interest. 

Recently, a sufficient condition for the two--weight boundedness of sparse 
operators was provided by Treil--Volberg and Lacey--Spencer \cites{TreVol,LacSpe}. 
The conditions are in terms of so--called ``entropy bumps'' introduced in 
\cite{TreVol}. We give an efficient proof of these results and also give a 
new condition. The main results are as follows. Throughout, let 
\begin{align*}
\rho_{\sigma}(Q):=\frac{1}{\sigma(Q)}\int_{Q}M(\unit_{Q}\sigma)(x)dx,
\hspace{.2in}
[w,\sigma]_{p,\eps_p}
:=\sup_{Q\textnormal{ a cube}}
  \avg{w}_{Q}\avg{\sigma}_{Q}^{p-1}\rho_{\sigma}(Q)\eps_p(\rho_\sigma(Q)).
\end{align*}
where $\eps_p$ is an increasing function on $[1,\infty]$ satisfying 
$\sum_{r\in\mathbb{N}}\eps_p(2^{r})^{-\frac{1}{p}}<\infty$. 
Our first theorem is the following
\begin{thm}\label{main}
With definitions as above, and $1<p<\infty$, there holds
\begin{align*}
\norm{T_\mathcal{S}\sigma\cdot:L^p(\sigma)\to L^q(w)}
\lesssim [w,\sigma]_{p,\eps_p}^{\frac{1}{p}}+[\sigma,w]_{p'\eps_{p'}}^{\frac{1}{p'}}.
\end{align*}
\end{thm}

For our next theorem, let 
\begin{align*}
[[w,\sigma]]_{p,\alpha_p}
:=\sup_{Q\textnormal{ a cube}}\avg{w}_{Q}\avg{\sigma}_{Q}^{p-1}\alpha_p(\avg{\sigma}_Q),
\end{align*}
where $\alpha_p$ is a function that is decreasing on $(0,1)$ and increasing on 
$(1,\infty)$ and that satisfies $\sum_{r\in\mathbb{Z}}\alpha_p(2^{-r})^{-\frac{1}{p}}<\infty$.
We have
\begin{thm}\label{main2}
With definitions and above, and $1<p<\infty$ there holds
\begin{align*}
\norm{T_\mathcal{S}\sigma\cdot:L^p(\sigma)\to L^p(w)}
\lesssim [[w,\sigma]]_{p,\alpha_p}^{\frac{1}{p}}
  +[[\sigma,w]]_{p',\alpha_{p'}}^{\frac{1}{p'}}.
\end{align*}
\end{thm}

The type of theorems we are proving are known as ``bumps'', because they slightly 
strengthen the joint $A_p$ characteristic. The bumps in Theorem \ref{main} were 
introduced in \cite{TreVol} and are known as ``entropy bumps'', and the bumps in 
Theorem \ref{main2} seem to be new. There is a long history of theorems 
of this type (see for example \cites{CruMoe2013,CruMarPer2011,CruMarPer2007,CruRezVol2014,Lac2016,HytLac2012,
LerMoe2013,Neu1983,NazRezTreVol2013,Per1994}) but 
in \cite{TreVol} it is shown that under some mild conditions, the entropy bumps are 
smaller than other bumps and so this approach is more robust.

Our proof builds on the ideas in \cite{LacSpe} and uses an interesting formula by 
Hyt\"onen in \cite{Hyt} that 
generalizes the expansion of sums like $(\sum_j a_j)^2$ to powers other 
than $2$. This formula seems to be powerful and it seems to have been first 
observed in \cite{Hyt}. 


\section{The Proof}
The proof will use the following facts and notation. First, all cubes 
considered below are in the sparse collection $\mathcal{S}$. For a collection $\mathcal{Q}$
of cubes let $[w,\sigma]^\mathcal{Q}_{p}:=\sup_{Q\in\mathcal{Q}}
  \avg{w}_{Q}\avg{\sigma}_{Q}^{p-1}$.
The first fact we use is the following deep theorem, 
originally due to Sawyer. See \cites{Saw, Hyt, LacSawUri}.
\begin{lm}\label{testingdyad}
Let $\md$ be a dyadic grid and let $\m{S}\subset \md$ be sparse. 
Define:
\begin{align*}
\mathcal{T}_1 :=
\sup_{P\in\m{S}}\frac{1}{\sigma(P)}\int_{P}
    \abs{\sum_{Q\in\m{S}:Q\subset P}\avg{\sigma}_Q\one_Q(x)}^pw(x)dx,
\end{align*}
\begin{align*}
\mathcal{T}_2 :=
\sup_{P\in\m{S}}\frac{1}{w(P)}\int_{P}
    \abs{\sum_{Q\in\m{S}:Q\subset P}\avg{w}_Q\one_Q(x)}^{p'}\sigma(x)dx.
\end{align*}
Then:
\begin{align*}
\norm{T_\mathcal{S}\sigma\cdot:L^p(\sigma)\to L^p(w)}\lesssim \mathcal{T}_1^{\frac{1}{p}}
  +\mathcal{T}_2^{\frac{1}{p'}}.
\end{align*}
\end{lm} 

We first give the proof of Theorem \ref{main} in the case $p=2$. 
\begin{proof}[Proof of Theorem \ref{main} when $p=2$]
We will verify the testing conditions hold; we will only verify the first condition 
as the second condition is verified similarly. Fix $P \in \mathcal{S}$. By the triangle inequality and the summability condition of $\eps_2$, it suffices to show
\begin{align}\label{E:testing}
\int_{P}\abs{\sum_{Q\in\mathcal{Q}_r}
  \avg{\sigma}_Q\unit_Q}^2w
\lesssim \frac{1}{\eps_2(2^r)}[{\sigma},{w}]_{2,\eps_{2}}\sigma(P),
\end{align}
where $\mathcal{Q}_{r}:=\{Q:Q\subset P \textnormal{ and }
  \rho_{\sigma}(Q)\simeq
  2^{r}\}$ for $r\in\mathbb{N}$. Since two cubes in $\mathcal{Q}_r$ are either nested or disjoint, there holds
\begin{align*}
\abs{\sum_{Q\in\m{Q}_r}\avg{\sigma}_Q\unit_Q(x)}^2
\simeq \sum_{Q\in\mathcal{Q}_r}\sum_{Q'\subset Q}
  \avg{\sigma}_{Q}\avg{\sigma}_{Q'}\unit_{Q'}(x).
\end{align*}
Inserting this into \eqref{E:testing}, and using $\rho_{\sigma}(Q)\simeq 2^r$ for $Q\in\mathcal{Q}_r$, 
\begin{align*}
\int_{P}\abs{\sum_{Q\in\m{Q}_r}\avg{\sigma}_Q\unit_Q}^2w
&\simeq \sum_{Q\in\mathcal{Q}_r}\sum_{Q'\subset Q}
   \avg{\sigma}_{Q}  \avg{\sigma}_{Q'}w(Q')
\\&=\sum_{Q\in\mathcal{Q}_r}\avg{\sigma}_{Q}
  \sum_{Q'\subset Q}
  \abs{Q'}
  \avg{\sigma}_{Q'}\avg{w}_{Q'}\frac{\rho_{\sigma}(Q)\eps(\rho_{\sigma}(Q))}
    {\rho_{\sigma}(Q)\eps(\rho_{\sigma}(Q))}.
\\&\lesssim \frac{1}{2^r\eps_2(2^r)}[{\sigma},{w}]_{2,\eps_2}\sum_{Q\in\mathcal{Q}_r}\avg{\sigma}_{Q}
  \sum_{Q'\subset Q}
  \abs{Q'}.
\end{align*}
Since $\mathcal{Q}_r$ is sparse, $\sum_{Q\in\mathcal{Q}_r}\avg{\sigma}_{Q} \sum_{Q'\subset Q}\abs{Q'} \lesssim \sum_{Q\in\mathcal{Q}_r}\sigma(Q).$

Set $\mathcal{Q}^*_r$ to be the maximal cubes in $\mathcal{Q}_r$. Using the fact that $\abs{E_Q}
\simeq\abs{Q}$ and that $\{E_Q\}$ are pairwise disjoint, there holds:
\begin{align}\label{E:ainfty}
\sum_{Q\in\mathcal{Q}_r}\sigma(Q)
\simeq \sum_{Q^*\in\mathcal{Q}^*_r} \int_{Q^*}\sum_{Q\subset Q^*}\avg{\sigma}_Q\unit_{E_Q}
\leq \sum_{Q^*\in\mathcal{Q}^*_r} \int_{Q^*} M(\sigma\unit_{Q^*})
\leq 2^r \sum_{Q^*\in\mathcal{Q}^*_r}\sigma(Q^*).
\end{align}
Since the cubes in $\mathcal{Q}_r^*$ are pairwise disjoint, the sum is bounded by $\sigma(P)$, as desired.
\end{proof}

To use a similar idea for $p\neq 2$ we need the following lemma proven in \cite{Hyt}.
\begin{lm}\label{hytlm}
Let $\mathcal{Q}$ be any collection of cubes. With obvious notation 
there holds
\begin{align}\label{E:exp1}
\int_{P}\left(\sum_{Q\in\mathcal{Q}:Q\subset P}\avg{\sigma}_{Q}\unit_{Q}\right)^{p}w
\lesssim [w,\sigma]^\mathcal{Q}_{p}\sum_{Q\subset P}\avg{\sigma}_{Q}\abs{Q}.
\end{align}
\end{lm}

We use this to prove Theorem \ref{main} for all $p>1$. 
\begin{proof}[Proof of Theorem \ref{main}]
We will verify the testing conditions. For simplicity, we will verify the 
first condition; the dual condition is verified similarly. Thus, let 
$P$ be any cube in $\mathcal{D}$. 
For $r\geq 0$, let $\mathcal{Q}_r=\{Q\subset P:\rho_\sigma(Q)\simeq 2^r\}$. Note that 
for these cubes, there holds
\begin{align*}
[w,\sigma]^\mathcal{Q}_{p}
\lesssim \frac{1}{2^r\eps_p(2^r)}[w,\sigma]_{p,\eps_p}.
\end{align*}
Therefore, by the triangle inequality and Lemma \ref{hytlm} there holds
\begin{align*}
\left(\int_{P}\left(
  \sum_{Q:Q\subset P}\avg{\sigma}_{Q}
  \unit_{Q}\right)^{p}w\right)^{\frac{1}{p}}
&\leq \sum_{r\geq 0}\left(\int_{P}\left(
  \sum_{Q\in\mathcal{Q}_r}\avg{\sigma}_{Q}
  \unit_{Q}\right)^{p}w\right)^{\frac{1}{p}}
\\&\lesssim[w,\sigma]_{p,\eps_p}^{\frac{1}{p}}
  \sum_{r\geq 0}\frac{1}{\eps_p(2^r)^{\frac{1}{p}}}
  \left(\frac{1}{2^r}\sum_{Q\in\mathcal{Q}_r}\sigma(Q)\right)^{\frac{1}{p}}
\\&\lesssim[w,\sigma]_{p,\eps_p}^{\frac{1}{p}}
  \sum_{r\geq 0}\frac{1}{\eps_p(2^r)^{\frac{1}{p}}}
  \sigma(P)^{\frac{1}{p}}.
\end{align*}
In the last estimate, we used the fact that for the cubes in $\mathcal{Q}_r$, 
$\rho_\sigma(Q)\simeq 2^r$ and so we can use the same estimate as in 
\eqref{E:ainfty}.
The summability condition on $\eps_p$ completes the proof. 
\end{proof}

We conclude with the proof of Theorem \ref{main2}.
\begin{proof}[Proof of Theorem \ref{main2}]
As above we need to verify the testing conditions, and we will only verify 
the first. Thus, let $P$ be any cube in $\mathcal{D}$. 
For $r\in\mathbb{Z}$ let $\mathcal{Q}_r=\{Q\subset P:\avg{\sigma}_{Q}\simeq 2^r\}$.
Using the summability condition on $\alpha_p$, as in the proof of Theorem 
\ref{main} we may assume that all cubes are contained in $\mathcal{Q}_r$. 

Again let $\mathcal{Q}_r^\ast$ denote the maximal cubes in $\mathcal{Q}_r$.  Using Lemma \ref{hytlm}, there holds
\begin{align*}
\int_{P}\left(\sum_{Q\in\mathcal{Q}_r}\avg{\sigma}_{Q}\unit_{Q}\right)^{p}w
&\lesssim \frac{1}{\alpha_p(2^r)}[[w,\sigma]]_{p,\alpha_p}
  \sum_{Q\in\mathcal{Q}_r}\avg{\sigma}_{Q}\abs{Q}
\\&\simeq\frac{1}{\alpha_p(2^r)}[[w,\sigma]]_{p,\alpha_p}\sum_{Q^\ast\in\mathcal{Q}_r^\ast}
  \sum_{Q\subset Q^\ast}\abs{Q}
\\&\simeq\frac{1}{\alpha_p(2^r)}[[w,\sigma]]_{p,\alpha_p}
  \sum_{Q^\ast\in\mathcal{Q}_r^\ast}\abs{Q}^\ast.
\end{align*}
In the second line we used the definition of $\mathcal{Q}_r$ and in the third 
line we used sparseness. Again, using the definition of $\mathcal{Q}_r$, the 
sum is equivalent to $\sum_{Q^\ast\in\mathcal{Q}_r^\ast}\sigma({Q}^\ast)$ and 
by the maximality of the cubes in $\mathcal{Q}^\ast$, it follows that this sum 
is dominated by $\sigma(P)$.
\end{proof}

\textbf{Acknowledgment.}R. Rahm would like to thank Kabe Moen for telling him 
about the formula from \cite{Hyt}.

\end{document}